\documentclass[11pt]{amsart}%
\usepackage{palatino, mathpazo}
\usepackage{amssymb}
\usepackage{amsfonts}
\usepackage{graphicx}
\usepackage{amsmath}
\usepackage[shortlabels,inline]{enumitem}
\usepackage{hyperref}%
\usepackage{cite}
\providecommand{\U}[1]{\protect\rule{.1in}{.1in}}
\hypersetup{
	colorlinks=true,
	linkcolor=blue,
	filecolor=black,
	urlcolor=black,
	citecolor=red,
}
\newtheorem{theorem}{Theorem}[section]

\newtheorem{proposition}[theorem]{Proposition}
\newtheorem{remark}[theorem]{Remark}

\numberwithin{equation}{section}
\usepackage{caption}
\makeatletter
\def\fps@figure{htbp} 
\def\fnum@figure{\textbf{Figure. \thefigure}} 
\makeatother

\begin{document}

\title[Even Cone Spherical Metrics: Blow-Up at a Cone Singularity]{Even Cone Spherical Metrics: Blow-Up at a Cone Singularity.}

\author{Ting-Jung Kuo}
\address[Ting-Jung Kuo]{Department of mathematics, National Taiwan Normal University, Taipei, 11677, Taiwan \& National Center for Theoretical Sciences, No.1 Sec.4 Roosevelt Rd., National Taiwan University, Taipei 10617, Taiwan.}
\email{tjkuo1215@ntnu.edu.tw, tjkuo1215@gmail.com}

\author{Xuanpu Liang}
\address[Xuanpu Liang]{Laboratory of Mathematics and Complex Systems (Ministry of Education), School of Mathematical Sciences, Beijing Normal University, Beijing 100875, China.}
\email{xuanpuliang@mail.bnu.edu.cn}

\author{Ping-Hsiang Wu}
\address[Ping-Hsiang Wu]{Department of Mathematics, National Taiwan Normal University, Taipei 11677, Taiwan.}
\email{phsiangwu@gmail.com}

\begin{abstract}
We study families of spherical metrics on the flat torus $E_{\tau}$ $=$
$\mathbb{C}/\Lambda_{\tau}$ with blow-up behavior at prescribed conical
singularities at $0$ and $\pm p$, where the cone angle at $0$ is $6\pi$, and
at $\pm p$ is $4\pi$. We prove that the existence of such a necessarily
unique, even family of spherical metrics is completely determined by the
geometry of the torus: such a family exists if and only if\textbf{ }the Green
function $G(z;\tau)$ admits a pair of nontrivial critical points $\pm a$. In
this case, the cone point $p$ must equal $a$, and the corresponding monodromy
data is $\left(  2r,2s\right)  $, where $a=r+s\tau.$

An explicit transformation relating this family to the one with a single
conical singularity of angle $6\pi$ at the origin is established in Theorem
\ref{Main theorem4}. A rigidity result for rhombic tori is proved in Theorem
\ref{Main theorem3}.

\end{abstract}
\maketitle
\tableofcontents


\section{Introduction}

\label{Introduction}

Let $E_{\tau}:=\mathbb{C}/\Lambda_{\tau}$, where $\Lambda_{\tau}$ $=$
$\mathbb{Z}\oplus\tau\mathbb{Z}$ and $\tau\in$ $\mathbb{H}$ $=$ $\left\{
\tau|\operatorname{Im}\tau>0\right\}  $, denote the flat torus. The study of
cone spherical metrics on $E_{\tau}$, i.e. metrics of constant Gaussian
curvature $1$ with prescribed conical singularities and angles, has been a
central problem in conformal geometry. Significant progress has been made in
this direction; see \cite{CLW, CCChen-Lin 1,Chen-Lin-sharp
nonexistence,Chen-Kuo-Lin-8pi+8pi, Bergweiler-Eremenko-dynamics,
Chen-Fu-Lin-Hitchin, Chen-Kuo-Lin-16pi, CWWXu, Eremenko-.Gabrielov-on metrics,
Eremenko-Mondello-Panov, EGMP, LW-AnnMath, LW, LSXu}.

Let $p\in E_{\tau}$. Throughout this paper, we assume
\begin{equation}
p\in E_{\tau}\setminus E_{\tau}[2], \label{Assumption 1}%
\end{equation}
where $E_{\tau}[2]$ denotes the 2-torsion points on $E_{\tau}$. That is
$E_{\tau}[2]$ $=$ $\{\frac{\omega_{k}}{2},k$ $=$ $0,1,2,3.\}$ Here $\omega
_{0}=0,\omega_{1}=1,\omega_{2}=\tau,\omega_{3}=1+\tau.$

We investigate the following curvature equation:%
\begin{equation}
\Delta u+e^{u}=8\pi\delta_{0}+4\pi(\delta_{p}+\delta_{-p})\text{ on }E_{\tau
}\text{,} \label{Curvature 1}%
\end{equation}
where $\delta_{q}$ denotes the Dirac measure at $q\in E_{\tau}$. Any solution
$u$ gives rise to the conformal metric $\frac{1}{2}e^{u}dz^{2}$, which has
constant positive curvature $1$ on the elliptic curve $E_{\tau}$, with a
conical singularity of cone angle $6\pi$ at $0$, and conical singularities of
cone angle $4\pi$ at $\pm p$.

Physically, equation \eqref{Curvature 1}$_{\tau,p}$ appears in statistical mechanics as
the mean-field equation for the Euler flow in Onsager's vortex model
\cite{CLMP}, and it is also connected to the study of self-dual condensates in
the Chern--Simons--Higgs model of superconductivity\ \cite{LY, NT}.

The curvature equation \eqref{Curvature 1}$_{\tau,p}$ is integrable equation in the sense
of the classical Liouville Theorem, which asserts that any solution $u(z)$ of
equation \eqref{Curvature 1}$_{\tau,p}$ is given as the form
\begin{equation}
u(z)=\log\frac{8|f^{\prime}(z)|^{2}}{(1+|f(z)|^{2})^{2}}. \label{502}%
\end{equation}
See \cite{CLW, Prajapat-Tarantello} for a proof. Here, $f(z)$ is locally
meromorphic and is commonly referred to as the developing map in the literature.

When the total cone angle is an odd multiple of $2\pi$, the developing map
$f(z)$ associated with a solution $u(z)$ must satisfy type II condition:%
\begin{equation}
f(z+1)=e^{-4\pi is}f(z)\text{ and }f(z+\tau)=e^{4\pi ir}f(z) \label{type II}%
\end{equation}
for some real pair $\left(  r,s\right)  \in\mathbb{R}^{2}$; see \cite{CLW} for
a proof. Under this condition (\ref{type II}), any solution $u(z)$ to the
curvature equation \eqref{Curvature 1}$_{\tau,p}$ gives rise to a one-parameter family of
solutions
\begin{equation}
u_{\beta}(z)=\log\frac{8e^{2\beta}\left\vert f^{\prime}\right\vert ^{2}%
}{\left(  1+e^{2\beta}\left\vert f\right\vert ^{2}\right)  ^{2}},\text{ }%
\beta\in\mathbb{R}\text{,} \label{family}%
\end{equation}
which corresponds to replacing $f$ by $f_{\beta}:=e^{\beta}f$. The real pair
$\left(  r,s\right)  $ is referred to as the \textbf{monodromy data }of the
family $\left\{  u_{\beta}(z;\tau,p)|\beta\in\mathbb{R}\right\}  $. When
necessary, we denote the family as $\left\{  u_{\beta}(z;\tau,p)\right\}  $ or
$\left\{  u_{\beta}(z;\tau,p)\right\}  _{(r,s)}$ to emphasize the underlying monodromy.

We say that the family is an \textbf{even} \textbf{family} if it contains an
even solution, which is necessarily unique. Equivalently, there exists $\beta$
such that $f_{\beta}(-z)$ $=$ $\frac{1}{f_{\beta}(z)}$. Otherwise, it is said
to be an \textbf{non-even family}.

The parameter $\beta$ $\in$ $\mathbb{R}$ serves as a scaling parameter along a
one-parameter family in the moduli space, under which the geometry of the
solution degenerates into delta-mass configurations characterized by curvature
concentration. Specifically, the family $\left\{  u_{\beta}(z;\tau,p)\right\}
$ exhibits blow-up behavior: as $\beta\rightarrow+\infty$ , blow-up occurs at
the zeros of $f$, and as $\beta\rightarrow-\infty$, at its poles.

In our previous work \cite[Corollary 3.10]{KLW}, we proved that when $E_{\tau
}$ is the \textbf{rhombic torus}, i.e., $\tau=\rho=e^{\pi i/3}$, the curvature
equation \eqref{Curvature 1}$_{\tau,p}$ with%
\[
\tau=\rho=e^{\pi i/3}\text{ and }p=\frac{1+\rho}{3}%
\]
admits a \textbf{unique even family} of solutions
\begin{equation}
\left\{  \mathbf{u}_{\beta}(z;\rho,\frac{1+\rho}{3})\right\}  \label{example}%
\end{equation}
which exhibit blow-up behavior at the cone singularities
\[
\pm\frac{1+\rho}{3}\text{ mod }\Lambda_{\rho}%
\]
as $\beta\rightarrow\pm\infty$, respectively.

Mathematically, the blow-up of a cone spherical metric at a cone point
corresponds to a concentration of curvature that disrupts the regular
geometric structure, as the metric diverges to infinity at the singularity.

A natural question arises:

	\textbf{Question }: \textit{For which pairs }$(\tau,p),$\textit{ with }%
	$\tau\in\mathbb{H}$\textit{ and }$p\in E_{\tau},$\textit{ satisfying
		(\ref{Assumption 1}), does the the curvature equation (\ref{Curvature 1}%
		)}$_{\tau,p}$\textit{ admit an even family of solutions $\{\textbf{u}_{\beta}(z;\tau,p)\}$ that blow up behavior
		at the cone singularities }$\pm p$\textit{ as }$\beta\rightarrow\pm\infty
	$\textit{, respectively?}
	
	Since each $\textbf{u}_{\beta}$ exhibits a conical singularity at $p$, specifically,
	\[
	\textbf{u}_{\beta}(z)=2\log|z-p|+O(1),
	\]
    where the bounded term $O(1)$ depends on $\beta$, the blow-up at this point is to be understood as occurring at $p$ fter the conical singularity has been removed through the corresponding Green function. That is,
	\[
	\textbf{u}_{\beta}+4\pi\, G(z-p;\tau)
	\]
	blows up at $p$ as $\beta\to +\infty$.

To address the above question , we first introduce the Green function
$G(z;\tau)$ on $E_{\tau}$, defined by
\[
-\Delta G(z;\tau)=\delta_{0}-\frac{1}{\left\vert E_{\tau}\right\vert }\text{,
}\int_{E_{\tau}}G(z;\tau)=0,
\]
where $\left\vert E_{\tau}\right\vert $ denotes the area of the torus. The
function $G(z;\tau)$ is even and has a unique singularity at $z=0$ mod
$\Lambda_{\tau}$. By symmetry, it follows that
\[
\nabla G(\frac{\omega_{k}}{2};\tau)=0,k=1,2,3.
\]
That is, each half-period $\frac{\omega_{k}}{2}$ is always a critical point of
$G(z;\tau)$, and these are referred to as the \textbf{trivial critical
points}. As shown in \cite{LW-AnnMath}, for any torus $E_{\tau}$, the Green
function $G(z;\tau)$ either has exactly three trivial critical points, or it
additionally admits a pair of \textbf{nontrivial critical points} $\pm a\in
E_{\tau}$, where
\[
a=r+s\tau,\text{ }\left(  r,s\right)  \in\mathbb{[}0,1\mathbb{]}^{2}%
\setminus\frac{1}{2}\mathbb{Z}^{2}.
\]

Since nontrivial critical points always appear as a pair $\pm a$, we may,
without loss of generality, assume%
\begin{equation}
a=r+s\tau\text{, }\left(  r,s\right)  \in\lbrack0,1/2]\times\lbrack
0,1]\setminus\frac{1}{2}\mathbb{Z}^{2}\text{.}
\label{nontrivial critical point}%
\end{equation}
The existence of nontrivial critical points $\pm\sigma$ in the Green function
corresponds to flat tori that admit a necessarily unique, even family of
solutions $\{v_{\beta}(z;\tau)\}$ to the following curvature equation:%

\begin{equation}
\Delta v+e^{v}=8\pi\delta_{0}\text{ on }E_{\tau}. \label{Curvature single}%
\end{equation}
This family has monodromy data $\left(  r,s\right)  $ and exhibits blow-up
behavior at $\pm a$ \textit{as }$\beta\rightarrow\pm\infty$,
respectively,\ where $a=r+s\tau$; see \cite{LW-AnnMath}.

We are now ready to state our main result:

\begin{theorem}
\label{Main theorem1} Let $\tau\in\mathbb{H}$. Then there exists $p\in
E_{\tau}\setminus E_{\tau}[2]$ such that the curvature equation
\textit{(\ref{Curvature 1})}$_{\tau,p}$ admits an even family of solutions
$\{\mathbf{u}_{\beta}(z;\tau,p)\}$, blowing up at the cone singularities $\pm
p$\textit{ as }$\beta\rightarrow\pm\infty$, respectively,\textit{ if and only
if the Green function }$G(z;\tau)$ admits a pair of nontrivial critical point $\pm a$. Moreover, in this case, the even family is unique, and $p=a$. Equivalently,%
\[
\left\{  \mathbf{u}_{\beta}(z;\tau,p)\right\}  =\left\{  \mathbf{u}_{\beta
}(z;\tau,a)\right\}.
\]
\end{theorem}

\begin{remark}
Surprisingly, Theorem \ref{Main theorem1} reveals that the existence of such
an even blow-up family is governed entirely by the intrinsic geometry of the
torus. Specifically, the moduli parameter $\tau$ and the cone singularity
positions $\pm p$ are completely determined by the structure of the Green
function $G(z;\tau)$ on the torus.
\end{remark}

\textbf{Assume} $G(z;\tau)$ \textbf{admits a pair of nontrivial critical
points.} According to Theorem \ref{Main theorem1}, these two even families
$\{\mathbf{u}_{\beta}(z;\tau,a)\}$ and $\left\{  v_{\beta}(z;\tau)\right\}  $
both exhibit the \textbf{same} blow-up behavior at $\pm a$ \textit{as }%
$\beta\rightarrow\pm\infty$, respectively. However, they correspond to two
distinct families of cone spherical metrics on $E_{\tau}$: the former has cone
angles $6\pi$ at $0$ and $4\pi$ at $\pm a$, whereas the latter has a single
conical singularity of angle $6\pi$ at $0$ only.

This naturally raises the question:\textbf{ What is the monodromy data of the
even family} $\left\{  \mathbf{u}_{\beta}(z,\tau,a)\right\}  $
\textbf{obtained in Theorem \ref{Main theorem1}}? Our second main result
provides the answer.

\begin{theorem}
\label{Main theorem2}Assume $G(z;\tau)$ admits a pair of nontrivial critical
points $\pm a$ satisfying (\ref{nontrivial critical point}). Then the
monodromy data of the unique even family $\left\{  \mathbf{u}_{\beta}%
(z;\tau,a)\right\}  $ obtained in Theorem \ref{Main theorem1}, for the
curvature equation (\ref{Curvature single}) is precisely $\left(
2r,2s\right)  $ mode $\mathbb{Z}^{2}$.
\end{theorem}

Since both $\{\mathbf{u}_{\beta}(z;\tau,a)\}$ and $\left\{  v_{\beta}%
(z;\tau)\right\}  $ are even families, without loss of generality, we may
denote by $\mathbf{u}_{0}(z;\tau,a)$ and $v_{0}(z;\tau)$ the unique even
solutions within each family. Let $f(z;\tau,a)$ and $h(z;\tau)$ be their
respective type II developing maps, so that
\begin{equation}
\mathbf{u}_{0}(z;\tau,a)=\log\frac{8|f^{\prime}(z;\tau,a)|^{2}}{(1+|f(z;\tau
,a)|^{2})^{2}}\text{, }v_{0}(z;\tau)=\log\frac{8|h^{\prime}(z;\tau)|^{2}%
}{(1+|h(z;\tau)|^{2})^{2}}. \label{type II developing map}%
\end{equation}
Our third main theorem establishes the precise transformation relating these
two even families.

\begin{theorem}
\label{Main theorem4} Assume that the Green function $G(z;\tau)$ admits a pair of nontrivial critical points $\pm a$ satisfying \eqref{nontrivial critical point}. Let
$\mathbf{u}_{0}(z;\tau,a)$ and $v_{0}(z;\tau)$ denote the unique even
solutions in the families $\{\mathbf{u}_{\beta}(z;\tau,a)\}$ and $\left\{
v_{\beta}(z;\tau)\right\}$, respectively, as obtained in Theorem
\ref{Main theorem1} and the single-source equation (\ref{Curvature single}).
Suppose $f(z;\tau,a)$ and $h(z;\tau)$ are their type II developing maps as in
(\ref{type II developing map}). Then
\[
f(z;\tau,a)=h(z;\tau)^{2}.
\]
This identity establishes an explicit correspondence between the two even families. 
\end{theorem}

Recently, it was shown in \cite[Theorem 3.3]{KLW} that the curvature equation%
\begin{equation}
\Delta u+e^{u}=8\pi\delta_{0}+4\pi(\delta_{a}+\delta_{-a})\text{ on }E_{\tau},
\label{curvature special}%
\end{equation}
admits another family of solutions, denoted by $\left\{  \mathfrak{u}_{\beta
}(z;\tau,a)\right\}  _{(r,s)}$, which shares the same monodromy data $\left(
r,s\right)  $ as $\left\{  v_{\beta}(z;\tau)\right\}  _{(r,s)}$ arising from
the single-source equation (\ref{Curvature single}). Moreover, 
\[
\left\{
\mathfrak{u}_{\beta}(z;\tau,a)\right\}  _{(r,s)}\text{ is an even family of \eqref{Curvature 1}$_{\tau,a}$ }\Longleftrightarrow\,\wp(a;\tau)=0;
\]
see  Theorem 3.3-(ii) and (3.17) in \cite{KLW} for $p_*=a=r+s\tau$.

We now state our final main result, which establishes a rigidity phenomenon.

\begin{theorem}
\label{Main theorem3}Assume $G(z;\tau)$ admits a pair of nontrivial critical
points $\pm a$ satisfying (\ref{nontrivial critical point}). Then for the
curvature equation (\ref{curvature special}), the two families
\[
\left\{  \mathbf{u}_{\beta}(z;\tau,a)\right\}  _{(2r,2s)}=\left\{
\mathfrak{u}_{\beta}(z;\tau,a)\right\}  _{(r,s)}%
\]
if and only if
\[
\tau=\rho=e^{\pi i/3},\text{ }\left(  r,s\right)  =\left(  \frac{1}{3}%
,\frac{1}{3}\right)  ,\text{ }a=\frac{1+\rho}{3}%
\]
In such case, the even family $\left\{  \mathfrak{u}_{\beta}(z;\tau
,a)\right\}  _{(1/3,1/3)}$ coincides with the explicit example presented in
(\ref{example}).
\end{theorem}

Accordingly, $\left\{  \mathfrak{u}_{\beta}(z;\tau,a)\right\}  _{(r,s)}$ is a
non-even family if $\left(  r,s\right)  \not =\left(  \frac{1}{3},\frac{1}%
{3}\right)  $. In such case, the blow-up configuration is also completely
determined by $a$; see \cite[Theorem 3.8]{KLW}. Thus, the non-even family
$\{\mathfrak{u}_{\beta}(z;\tau,a)\}_{(r,s)}$ and the even family $\left\{
\mathbf{u}_{\beta}(z;\tau,a)\right\}  _{(2r,2s)}$ represent two distinct types
of solutions to the same curvature equation (\ref{curvature special}). These
two families are linked via the even family $\left\{  v_{\beta}(z;\tau
)\right\}  _{(r,s)}$. 

Remarkably, while all three families are governed by the same underlying torus
geometry, they reflect it in fundamentally different ways. However, the
precise geometric distinction between $\left\{  \mathfrak{u}_{\beta}%
(z;\tau,a)\right\}  _{(r,s)}$ and $\{\mathbf{u}_{\beta}(z;$ $\tau
,a)\}_{(2r,2s)}$ remains to be clarified. In particular, the behavior of these
families as $a\rightarrow\frac{\omega_{k}}{2},$ for $k=0,1,2,3$ is not yet
fully understood. Clarifying the geometric implications of this distinction
represents a promising direction for further study.

\textbf{Organization of the Paper: }

In Section \ref{Integrability and Spectral Theory for the curvature equation}, we introduce the generalized Lam\'{e}-type equation associated
with the curvature equation (\ref{Curvature 1}) in the presence of an even
family of solutions. We formulate the corresponding spectral polynomial and
spectral curve, and investigate the complete reducibility of this class of
differential equations.

In Section \ref{The monodromy theory for the ODE}, we establish a precise criterion for the existence of an even
family of blow-up solutions at a cone singularity in terms of the structure of
the associated Lam\'{e}-type equation.

In Section \ref{The Proofs}, we present the proofs of Theorems \ref{Main theorem1},
\ref{Main theorem2}, and \ref{Main theorem3}.\vspace{5pt}

\textbf{Acknowledgement:} The authors are grateful to Professor Chin-Lung Wang for providing the Figure \ref{Omega}. Ting-Jung Kuo was supported by NSTC 113-2628-M-003-001-MY4. He is also grateful to the National Center for Theoretical Sciences (NCTS) for its constant support. 
\vspace{5pt}

\section{Integrability and Spectral Theory for the Curvature Equation}
\label{Integrability and Spectral Theory for the curvature equation}
Let $\wp(z;\tau)$ be the Weierstrass elliptic function with periods $\omega_{1}=1$ and $\omega_{2}=\tau$, defined by
	\[
	\wp\left(  z;\tau\right)  =\frac{1}{z^{2}}+\sum_{(m,n)\in\mathbb{Z}%
		^{2}\backslash(0,0)}\left[  \frac{1}{\left(  z-m-n\tau\right)  ^{2}}-\frac
	{1}{(m+n\tau)^{2}}\right]  .
	\]
	The Weierstrass zeta function is defined by
	\[
	\zeta\left(  z;\tau\right)  :=-\int^{z}\wp(\xi;\tau)\,d\xi,
	\]
     which is quasi-periodic, satisfying%
	\[
	\zeta\left(  z+\omega_{j};\tau\right)  =\zeta\left(  z;\tau\right)  +\eta
	_{j}(\tau)\text{, }j=1,2,
	\]
	where $\eta_{1}(\tau)$ and $\eta_{2}(\tau)$ are the corresponding quasi-periods.
	And let $\sigma(z;\tau)$ be the Weierstrass sigma function, given by
	\[
	\sigma(z;\tau):=\exp\int^{z} \zeta(\xi;\tau)d\xi,
	\]
    which satisfies the following transformation law:
    \begin{equation}
\sigma(z+\omega_{j};\tau)=-e^{\eta_{j}\left(  z+\frac{\omega_{j}}{2}\right)
}\sigma(z;\tau),\quad
\,j=1,2.\label{transformation law of the Weierstrass sigma function}%
\end{equation}
	For brevity, we denote $\wp(z;\tau)$, $\zeta(z;\tau)$ and $\sigma(z;\tau)$ simply by $\wp(z)$, $\zeta(z)$ and $\sigma(z)$, respectively. For further properties of these classical elliptic functions, see \cite{Akhiezer}.

The existence of solutions to the curvature equation \eqref{Curvature 1}$_{\tau,p}$ is closely related to the study of the monodromy of the following generalized Lam\'{e}-type equation:
\begin{equation}\label{GLE1,even}
y''(z)=q(z;A)\,y(z)\quad\text{on}\quad E_\tau,
\end{equation}
where the potential is given by
\begin{equation}
\label{even potential}q(z;A)=\left(
\begin{array}
[c]{l}%
2\wp(z)+\dfrac{3}{4}\left(  \wp(z+p)+\wp(z-p)\right) \\
+A\left(  \zeta(z+p)-\zeta(z-p)\right)  +B
\end{array}
\right),
\end{equation}
with $B$ determined by
\begin{equation}
\label{B in even potential}B=A^{2}-\zeta(2p)\,A-\dfrac{3}{4}%
\wp(2p)-2\wp(p).
\end{equation}
Since the potential $q(z;A)$ is an even elliptic function, this corresponds to seeking an even family of solutions to the curvature equation \eqref{Curvature 1}$_{\tau,p}$.

Under \eqref{B in even potential}, all singularities of \eqref{GLE1,even}$_{\tau,p,A}$ are \textbf{apparent}; namely, all local solutions are free of logarithmic singularities.

We briefly recall the monodromy representation of \eqref{GLE1,even}$_{\tau,p,A}$ under the apparent condition \eqref{B in even potential}.
Note that the local exponents at $z=0,\pm p$ are $-1,\,2$ and $-1/2,\,3/2$, respectively. Consequently, the corresponding local monodromy matrices are
\[
N_0=I_2\quad\text{and}\quad N_{\pm p}=-I_2.
\]
Fix a base point $z_0\in E_\tau\setminus\{0,\pm p\}$ such that the fundamental cycles $z_0\mapsto z_0+\omega_j$, $j=1,2$, do not intersect with the points $\{0,\pm p\}+\Lambda_\tau$. Let 
\[Y(z;z_0)=(y_1(z;z_0),\,y_2(z;z_0))^t\]
be a fundamental solution of the generalized Lam\'{e}-type equation \eqref{GLE1,even}$_{\tau,p,A}$ near $z_0$. The global monodromy matrices $M_j(p,A)$, $j=1,2$, are defined by analytic continuation along two fundamental cycles:
\[
Y(z+\omega_j;z_0)=M_j(p,A)\,Y(z;z_0),\quad j=1,2.
\]
These matrices $M_j(p,A)$, $j=1,2$, satisfy the following monodromy relation:
\[
M_1(p,A)\,M_2(p,A)=M_2(p,A)\,M_1(p,A).
\]
This leads to the following two cases:\\
\textbf{(i) Completely Reducible Case:} The matrices $M_j(p,A)$, $j=1,2$, can be simultaneously diagonalized, taking the form
\begin{equation}\label{completely reducible}
    M_1(p,A)=
\left(\begin{array}{cc}
    e^{-2\pi i s} & 0 \\
    0 & e^{2\pi i s}
\end{array}\right),\quad
M_2(p,A)=
\left(\begin{array}{cc}
    e^{2\pi ir} & 0 \\
    0 & e^{-2\pi ir}
\end{array}\right),
\end{equation}
for some $(r,s)\in\mathbb{C}^2$.\\
\textbf{(ii) Non-Completely Reducible Case:} The matrices $M_j(p,A)$, $j=1,2$, cannot be simultaneously diagonalized. In this situation, there exists $\mathcal{D}\in\mathbb{C}\cup\{\infty\}$ such that, up to a common conjugation, $M_j(p,A)$, $j=1,2$, can be represented by 
\begin{equation}\label{non-completely reducible1}
    M_1(p,A)=\pm
\left(\begin{array}{cc}
    1 & 0 \\
    1 & 1
\end{array}\right),\quad
M_2(p,A)=\pm
\left(\begin{array}{cc}
    1 & 0 \\
    \mathcal{D} & 1
\end{array}\right).
\end{equation}
In the case $\mathcal{D}=\infty$, this should be interpreted as
\begin{equation}\label{non-completely reducible2}
    M_1(p,A)=\pm
\left(\begin{array}{cc}
    1 & 0 \\
    0 & 1
\end{array}\right),\quad
M_2(p,A)=\pm
\left(\begin{array}{cc}
    1 & 0 \\
    1 & 1
\end{array}\right).
\end{equation}
The pair $(r,s)\in\mathbb{C}^2$ in Case (i) and $\mathcal{D}\in\mathbb{C}\cup\{\infty\}$ in Case (ii), are respectively called the monodromy data of equation \eqref{GLE1,even}$_{\tau,p,A}$.

The following theorem establishes the deep connection between the even family solutions of the curvature equation \eqref{Curvature 1}$_{\tau,p}$ and the generalized Lam\'{e}-type equation \eqref{GLE1,even}$_{\tau,p,A}$. 

\begin{theorem}[\textit{\cite[Theorem 3.1]{Chen-Kuo-Lin-Painleve VI}}]\label{pde to ode}
    Let $p\in E_{\tau}\setminus E_\tau[2]$. The curvature equation \eqref{Curvature 1}$_{\tau,p}$ on $E_\tau$ admits an \textbf{even} family of solutions if and only if there exists $A\in\mathbb{C}$, with $B$ given by \eqref{B in even potential}, such that the generalized Lam\'{e}-type equation \eqref{GLE1,even}$_{\tau,p,A}$ is completely reducible with unitary monodromy.
\end{theorem}

Theorem \ref{pde to ode} was proved in \cite{Chen-Kuo-Lin-Painleve VI}.
For the convenience of the readers, we briefly summarize the idea of the proof:
 \begin{proof}
  	For the necessity, suppose $u(z)$ is an even solution of \eqref{Curvature 1}$_{\tau,p}$. Consider the function $u_{zz}-\frac{1}{2}u_{z}^2$. By \eqref{Curvature 1}$_{\tau,p}$, it is easy to verify
		\[
		\left(u_{zz}-\frac{1}{2}u_{z}^2\right)_{\bar{z}}=0,\quad\forall z\notin\{0,\pm p\}.
		\]
		Since the local behavior of $u(z)$ at $0$ and $\pm p$ are given by
		
		\[	u(z)=4\ln|z|+O(1)\quad\text{near }0,\]
		and
		\[	u(z)=2\ln|z\mp p|+O(1)\quad\text{near }\pm p,\]
		we obtain that $u_{zz}-\frac{1}{2}u_{z}^2$ has double poles at $0$ and $\pm p$. Consequently, $u_{zz}-\frac{1}{2}u_{z}^2$ is an even elliptic function, which can be expressed in terms of classical elliptic functions as below:
		\begin{align}
			u_{zz} - \frac{1}{2}u_{z}^{2} = -2 \Biggl( & \, 2\wp(z) + \frac{3}{4}\left( \wp(z+p) + \wp(z-p) \right) \nonumber\\
			& + A\left( \zeta(z+p) - \zeta(z-p) \right) + B \Biggr)=-2q(z;A),
			\label{-2q(z)}
		\end{align}
		where
		\[
		A=\frac{1}{2}Res_{z=p}\left(u_{zz}-\frac{1}{2}u_{z}^2\right)=-\frac{1}{2}Res_{z=-p}\left(u_{zz}-\frac{1}{2}u_{z}^2\right),
		\]
		and $B\in\mathbb{C}$. We will show that $B$ is given by \eqref{B in even potential}.
		
		By Liouville Theorem, there exists a developing map $f(z)$ such that
		\begin{equation}
			u(z)=\log\frac{8|f^{\prime}(z)|^2}{\left(1+|f(z)|^2\right)^2}.\label{u(z) by Liouville Theorem}
		\end{equation}
		A direct computation shows that the Schwartzian derivative $\{f;z\}$ satisfies
		\begin{equation}
			\left\{f;z\right\}=u_{zz} - \frac{1}{2}u_{z}^{2}. \label{schwartz derivative}
		\end{equation}
		By \eqref{schwartz derivative} and \eqref{-2q(z)}, we obtain 
		\[
		\{f;z\}=-2q(z;A).
		\]
		A classical result then implies the existence of two linearly independent solutions   $y_1(z)$ and $y_2(z)$ of generalized Lam\'{e}-type equation \eqref{GLE1,even}$_{\tau,p,A}$ such that
		\begin{equation}
			f(z)=\frac{y_1(z)}{y_2(z)}.   \label{f=y1/y2}
		\end{equation}
		Define the Wronskian
		\[
		W=y_{1}^{\prime}(z)y_{2}(z)-y_{1}(z)y_{2}^{\prime}(z),
		\]
		which is a nonzero constant. Substituting \eqref{f=y1/y2} into \eqref{schwartz derivative} yields
		\[
		2\sqrt{2}W e^{-\frac{1}{2}u(z)}=|y_1(z)|^2+|y_2(z)|^2.
		\]
		Since $u(z)$ is single-valued and doubly periodic,  it follows  that the monodromy matrices with respect to $(y_1(z),y_2(z))$ lie in $SU(2)$, that is, the monodromy is unitary. This implies that the generalized Lamé-type equation \eqref{GLE1,even}$_{\tau,p,A}$ is completely reducible. 
		
		Note that if \eqref{GLE1,even}$_{\tau,p,A}$ has logarithmic singularities at $\pm p$,  the local monodromy matrix at $\pm p$ can not be in $\mathrm{SU}(2)$. Therefore, \eqref{GLE1,even}$_{\tau,p,A}$ is apparent at $\pm p$, which implies that $B$ is given by \eqref{B in even potential}.
		
		For the sufficiency, it suffices to prove the existence of even solution to \eqref{Curvature 1}$_{\tau,p}$. Assume there exist $A\in\mathbb{C}$ and $B$ as given by \eqref{even potential} such that \eqref{GLE1,even}$_{\tau,p,A}$ is completely reducible and has unitary monodromy. Then by Theorem \ref{monodromy+spectral main theorem}, there exists a fundamental solution $Y(z)=(y_1(z),y_2(z))^{t}$ such that the monodromy matrices $M_j(p,A)$ satisfy Case (i) with $(r,s)\in\mathbb{R}^2\setminus\frac{1}{2}\mathbb{Z}^2$. Actually, we may take $y_2(z)=y_1(-z)$.
		
		Define 
		\begin{equation}
			f(z):=\frac{y_1(z)}{y_2(z)}=\frac{y_1(z)}{y_1(-z)}.   \label{developing map: even case}
		\end{equation}
		Since the local exponent of $y_j(z)$ at $\pm p$ is either $-\frac{1}{2}$ or $\frac{3}{2}$, then the local exponent of $f$ at $\pm p$ is either $0$ or $\pm 2$, which implies $f(z)$ is single-valued near $z=\pm p$. Thus $f(z)$ can be extended to be a meromorphic function on $\mathbb{C}$, and satsfies type II condition \eqref{type II} for $(r,s)\in\mathbb{R}^2\setminus\frac{1}{2}\mathbb{Z}^2$. 
		
		Define
		\begin{equation}
			u(z):=\log\frac{8|f^{\prime}(z)|^2}{\left(1+|f(z)|^2\right)^2}.  \label{u(z) even's def}
		\end{equation}
		Then by \eqref{type II}, $u(z)$ is doubly periodic. From \eqref{developing map: even case}, we can derive
		\begin{equation}
			\left(\frac{f^{\prime\prime}}{f^{\prime}}\right)^{\prime}-\frac{1}{2}\left(\frac{f^{\prime\prime}}{f^{\prime}}\right)^{2}=-2q(z;A). \label{equ 09101}
		\end{equation} 
		Since the local exponent of $y_j(z)$ at $0$ is either $-1$ or $2$, then the local exponent $\rho$ of $f(z)$ at $0$  belongs to $\{0,\pm 3\}$. Then by \eqref{equ 09101}, we obtain
		\begin{equation*}
			f(z)=
			\begin{cases}
				f(0)+a_3 z^3+O(z^4),\quad    &a_3\neq 0,\quad\text{if }\rho=0;\\
				a_{-3}z^{-3}+O(z^{-2}),\quad &a_{-3}\neq0,\quad\text{if }\rho=-3;\\
				a_{3}^{\prime}z^3+O(z^4),\quad &a_{3}^{\prime}\neq 0,\quad\text{if }\rho=3.
			\end{cases}
		\end{equation*}
		Substituting into \eqref{u(z) even's def}, we see that
		\[
		u(z)=4\ln|z|+O(1)\quad\text{near }0.
		\]
		A similar analysis shows that
		\begin{equation*}
			u(z)=2\ln|z\mp p|+O(1)\quad\text{near }z=\pm p,
		\end{equation*}
		and 
		\[
		\Delta u+e^u=0\quad\text{on } E_{\tau}\setminus\{0,\pm p\}.
		\]
		Therefore,  $u(z)$ is a solution to \eqref{Curvature 1}$_{\tau,p}$ . Moreover, since $f(z)=\frac{1}{f(-z)}$ by \eqref{developing map: even case}, it follows that $u(z)=u(-z)$ is an even solution.
\end{proof}

In view of Theorem \ref{pde to ode}, two fundamental issues arise:
\begin{enumerate}
    \item [(1)] When does a generalized Lam\'{e}-type equation \eqref{GLE1,even}$_{\tau,p,A}$ admit completely reducible monodromy representation?
    \item [(2)] After (1), when is the monodromy representation unitary, up to a conjugation?
\end{enumerate}

To address (1), we need to develop the spectral theory of the generalized Lam\'{e}-type equation \eqref{GLE1,even}$_{\tau,p,A}$. We borrow the idea from elliptic Korteweg-De Vries (KdV) theory and these computations had been completed in \cite{KLW}. For the completeness of this paper, we briefly introduce them as below.

Consider the following second symmetric product of equation \eqref{GLE1,even}$_{\tau,p,A}$:
\begin{equation}
\Phi^{\prime\prime\prime}(z)-4q(z;A)\,\Phi^{\prime}(z)-2q^{\prime}%
(z;A)\,\Phi(z)=0\quad\text{on}\quad E_{\tau}.\label{3rd ODE}%
\end{equation}
The next theorem asserts the existence and uniqueness of the elliptic solution to equation \eqref{3rd ODE}.

\begin{theorem}[\textit{\cite[Theorem 4.5]{KLW}}]\label{unique elliptic sol, 3rd ode, even}
Up to a nonzero multiple, there exists a unique non-trivial elliptic solution $\Phi_{e}(z;A)$ of \eqref{3rd ODE} given by
\begin{equation}
\label{elliptic solution to 3rd ode, even case*}\Phi_{e}(z;A)=\left[
\begin{array}
[c]{l}%
\wp(z)+ \left(  \dfrac{A}{2}-\dfrac{3}{8}\dfrac{\wp^{\prime\prime}(p)}
{\wp^{\prime}(p)}\right)  \left(  \zeta(z+p)-\zeta(z-p)\right) \\[10pt]%
-A^{2}+\left(  \dfrac{\wp^{\prime\prime}(p)}{\,2\wp^{\prime}(p)\,}%
-\zeta(p)\right)  A\\[10pt]%
-\wp(p)+\dfrac{3}{4}\dfrac{\,\wp^{\prime\prime}(p)\,}{\wp^{\prime}(p)}%
\zeta(p)+\dfrac{3}{16}\dfrac{\,\wp^{\prime\prime}(p)^{2}\,}{\wp^{\prime
}(p)^{2}}%
\end{array}
\right].
\end{equation}
\end{theorem}
\normalsize
Define
\begin{equation}
    Q(z;A):=\dfrac{1}{2}\Phi_e(z;A)\,\Phi_e''(z;A)-\dfrac{1}{4}(\Phi_e'(z;A))^2-q(z;A)\,\Phi_e^2(z;A).
    \label{Q(A)'s definition}
\end{equation}
Since $\Phi_e(z;A)$ solves \eqref{3rd ODE}, it follows that $\frac{d}{dz}Q(z;A)=0$. Hence, $Q(z;A)$ is independent of $z$, we therefore write $Q(A)$. Furthermore, since $\Phi_e(z;A)$ and $q(z;A)$ are both  polynomials in $A$, $Q(A)$ is then a polynomial in $A$.

We refer to $Q(A)$ as the \textbf{spectral polynomial} associated with the generalized Lam\'{e}-type equation \eqref{GLE1,even}$_{\tau,p,A}$, which is computed explicitly in \cite[Theorem 4.5]{KLW}:
\begin{theorem}[\textit{\cite[Theorem 4.5]{KLW}}]\label{Q(A)'s computation}
The spectral polynomial \(Q(A)\) associated with \eqref{GLE1,even}$_{\tau,p,A}$ is explicitly given by
    \begin{equation}
\label{spectral polynomial in even case}Q(A)=-Y_{1}(A)\, Y_{2}(A),
\end{equation}
where
\begin{align*}
Y_{1}(A)  &  =A^{3} -\frac{5 \wp^{\prime\prime}(p)}{4\wp^{\prime}(p)}A^{2}
+3\left(  \wp(p) +\frac{1}{16}\frac{\wp^{\prime\prime}(p)^{2}}{\wp^{\prime
}(p)^{2}}\right)  A\\
&  \quad\quad+\frac{\wp^{\prime}(p)}{2} -\frac{9\wp^{\prime\prime}(p)}%
{4\wp^{\prime}(p)}\wp(p) +\frac{9}{64}\frac{\wp^{\prime\prime}(p)^{3}}%
{\wp^{\prime}(p)^{3}},\\
Y_{2}(A)  &  =A^{3} - \frac{\wp^{\prime\prime}(p)}{4\wp^{\prime}(p)}A^{2} -
\frac{ 5 }{16}\frac{\wp^{\prime\prime}(p)^{2}}{\wp^{\prime}(p)^{2}}A +
2\wp^{\prime}(p) - \frac{3}{64}\frac{\wp^{\prime\prime}(p)^{3}}{\wp^{\prime}(p)^{3}}.
\end{align*}
\end{theorem}

The \textbf{spectral curve} associated with $Q(A)$ is defined by
\begin{equation}\label{spectral curve, even case}
    \Gamma(\tau,p):=\left\{\,(A,C)\,|\,C^2=Q(A)\, \right\}.
\end{equation}
For $P=(A,C)\in\Gamma(\tau,p)$, denote its dual point by $P^*:=(A,-C)\in\Gamma(\tau,p)$.

We now introduce the Baker-Akhiezer
function at $P\in\Gamma(\tau,p)$. First, we define the meromorphic function
\begin{equation}\label{phi's def}
    \phi(P;z):=\dfrac{\,iC(P)+\frac{1}{2}\Phi_e'(z;A)\,}{\Phi_e(z;A)},\quad z\in\mathbb{C}.
\end{equation}
It is easy to verify that $\phi(P;z)$ satisfies the Riccati equation:
\begin{equation}\label{Riccati equation, even}
    \phi'(P;z)=q(z;A)-\phi^2(P;z).
\end{equation}
\begin{proposition}\label{0907prop1}
    The meromorphic function $\phi(P;z)$ has at most simple poles. Moreover, if $q_0$ is a pole of $\phi(P;z)$, then
    \begin{equation}
        \operatorname{Res}_{z=q_0}\phi(P;z)=
    \begin{cases}
        -1 & \text{if \,} q_0=0,\\
        -1/2 & \text{if \,} q_0=\pm p,\\
        ~1 & \text{if \,} q_0\notin\{0,\pm p\}.
    \end{cases}
    \end{equation}
\end{proposition}
The proof is same as the argument in \cite[Proposition 4.3]{KLW}, based on the Riccati equation \eqref{Riccati equation, even}.

Fix a base point $z_0\in\mathbb{C}\setminus\{0,\pm p\}$. For each $P\in\Gamma(\tau,p)$, define the \textbf{Baker-Akhiezer} function by
\begin{equation}\label{0907equ7}
\psi(P;z,z_0):=\exp\left(\int_{z_0}^z\phi(P;\xi)\,d\xi\right),\quad z\in\mathbb{C},
\end{equation}
where the integration path is chosen to avoid the singularities of the meromorphic function $\phi(P;\xi)$.

By Proposition \ref{0907prop1}, the Baker-Akhiezer function $\psi(P;z,z_0)$ is a multivalued meromorphic function with two local branches near $z=\pm p$. Moreover, by the Riccati equation \eqref{Riccati equation, even}, $\psi(P;z,z_0)$ solves the generalized Lam\'{e}-type equation \eqref{GLE1,even}$_{\tau,p,A}$ with parameter $A(P)$, where $A(P)$ denotes the $A$-coordinate of $P\in\Gamma(\tau,p)$.

Note that $A(P)=A(P^*)$. Consequently, both $\psi(P;z,z_0)$ and $\psi(P^*;z,z_0)$ solve the same equation \eqref{GLE1,even}$_{\tau,p,A}$ with parameter $A=A(P)=A(P^*)$. Therefore, $P,P^*\in\Gamma(\tau,p)$ represent the same generalized Lam\'{e}-type equation \eqref{GLE1,even}$_{\tau,p,A}$ with parameter $A(P)$.

We recall two basic identities for the Baker-Akhiezer functions $\psi(P;z,z_0)$ and $\psi(P^*;z,z_0)$ associated with \eqref{GLE1,even}$_{\tau,p,A}$ Firstly,
\begin{equation}\label{0907equ4}
    \psi(P;z,z_0)\,\psi(P^*;z,z_0)=\dfrac{\Phi_e(z;A)}{\Phi_e(z_0;A)}.
\end{equation}
Moreover, their Wronskian (with respect to $z$) is given by
\begin{equation}\label{0907equ5}
    W(\psi(P;z,z_0),\,\psi(P^*;z,z_0))=\dfrac{2iC}{\Phi_e(z_0;A)}.
\end{equation}
The proofs can be found in \cite{Kuo}.

Since different choices of the base point $z_0\notin\{0,\pm p\}$ change $\psi(P;z,z_0)$ only by a nonzero multiple, we henceforth write
\[
\psi(P;z,z_0)=\psi(P;z),\quad \psi(P^*;z,z_0)=\psi(P^*;z).
\]
As a consequence of \eqref{0907equ5}, we obtain the following theorem.
\begin{theorem}\label{Q(A)neq 0 iff BA are independent}
    Let $P=(A,C)\in\Gamma(\tau,p)$.
    The Baker-Akhiezer functions $\psi(P;z)$ and $\psi(P^*;z)$ is linearly independent if and only if $Q(A)\neq0$.  
\end{theorem}
By the definitions \eqref{phi's def} and \eqref{0907equ7}, together with the fact that $\Phi_{e}$ is elliptic, it follows that $\psi(P;z)$ is an elliptic function of the second kind; see, for instance, \cite{Kuo,KLW} for a proof. Hence there exist $\lambda_{j}(P)$, $j=1,2$ such that
\begin{equation}\label{0907equ8}
    \psi(P;z+\omega_j)=\lambda_j(P)\,\psi(P;z),\quad j=1,2.
\end{equation}
Moreover, by \eqref{0907equ4}, we have
\begin{equation}\label{0907equ9}
    \lambda_j(P^*)=\dfrac{1}{\,\lambda_j(P)\,}.
\end{equation}

Theorem \ref{Q(A)neq 0 iff BA are independent} implies that if $Q(A)\neq0$, then the generalized Lam\'{e}-type equation \eqref{GLE1,even}$_{\tau,p,A}$ is completely reducible. In fact, the converse also holds. The following theorem answers the issue (1).
\begin{theorem}\label{Q(A) neq 0 iff com.red.}
    For $P=(A,C)\in\Gamma(\tau,p)$, the generalized Lam\'{e}-type equation \eqref{GLE1,even}$_{\tau,p,A}$ is completely reducible if and only if $Q(A)\neq0$.
\end{theorem}
\begin{proof}
It suffices to prove the necessary part. Assume that \eqref{GLE1,even}$_{\tau,p,A}$ is completely reducible with $Q(A)=0$. By Theorem \ref{Q(A)neq 0 iff BA are independent}, both Baker-Akhiezer functions $\psi(P;z)$ and $\psi(P^*;z_0)$ are linearly dependent; namely,
\begin{equation}
\label{0806equ6}\psi(P;z)=c\,\psi(P^{*};z),\quad\text{for some $c\neq0$.}%
\end{equation}
Complete reducibility guarantees the existence of two linearly independent
solutions $y_{1}(z;A)$, $y_{2}(z;A)$ such that
\[
\left(
\begin{array}
[c]{c}%
y_{1}(z+\omega_{j};A)\\
y_{2}(z+\omega_{j};A)
\end{array}
\right)  = \left(
\begin{array}
[c]{cc}%
\widehat{\lambda_{j}}(P) & 0\\
0 & \dfrac{1}{\,\widehat{\lambda_{j}}(P)\,}%
\end{array}
\right)  \left(
\begin{array}
[c]{c}%
y_{1}(z;A)\\
y_{2}(z;A)
\end{array}
\right).
\]
Hence $y_{1}(z;A)\,y_{2}(z;A)$ is an elliptic solution of \eqref{3rd ODE}. By Theorem \ref{unique elliptic sol, 3rd ode, even} and
\eqref{0907equ4}, we obtain
\begin{equation}
\label{0806equ7}y_{1}(z;A)\,y_{2}(z;A)=\Phi_{e}(z;A)=\psi(P;z)\,\psi
(P^{*};z)=c\,\psi^{2}(P;z).
\end{equation}
It follows from \eqref{0806equ7} that $y_1(z;A)$ and $y_2(z;A)$ are linearly dependent, a contradiction. Hence $Q(A)\neq0$. 
\end{proof}

\section{The Monodromy Theory for the ODE}
\label{The monodromy theory for the ODE}

In this section, we study the monodromy theory of the generalized Lam\'{e}-type equation \eqref{GLE1,even}$_{\tau,p,A}$ via the Baker-Akhiezer function $\psi(P;z)$ at $P=(A,C)\in\Gamma(\tau,p)$. 

Recall from \eqref{0907equ8} that $\psi(P;z)$ is an elliptic function of the second kind. 
Define $(r(P),s(P))\in \mathbb{C}^2$ by
\begin{equation}\label{definition of r(P),s(P)}
    \lambda_1(P)=e^{-2\pi i s(P)}\quad\text{and}\quad
    \lambda_2(P)=e^{2\pi i r(P)}.
\end{equation}
Thus \((r(P),s(P))\) is defined modulo \(\mathbb{Z}^{2}\) (i.e., up to integer translations).
By \eqref{0907equ9}, we have
\begin{equation}\label{0907equ10}    
r(P^*)=-r(P)\quad\text{and}\quad s(P^*)=-s(P).
\end{equation}
We remark that if $(r(P),s(P))\notin \tfrac{1}{2}\mathbb{Z}^{2}$, at least one of $\lambda_{1}(P),\lambda_{2}(P)$ is not $\pm 1$. 
Without loss of generality, assume $\lambda_{1}(P)\neq \pm 1$, then $\lambda_{1}(P)$ and $\lambda_{1}(P^{*})$ are distinct eigenvalues of the monodromy matrices $M_j(p,A)$, $j=1,2$, which implies that $\psi(P;z)$ and $\psi(P^{*};z)$ are linearly independent. Hence, if $(r(P),s(P))\notin\frac{1}{2}\mathbb{Z}^2$, the generalized Lam\'{e}-type equation \eqref{GLE1,even}$_{\tau,p,A}$ is completely reducible. In fact, this condition is sufficient as well.

\begin{theorem}\label{monodromy+spectral main theorem}
Let $P=(A,C)\in\Gamma(\tau,p)$.
    Then the generalized Lam\'{e}-type equation \eqref{GLE1,even}$_{\tau,p,A}$ is completely reducible if and only if 
    $\left(r(P),s(P)\right)\notin\frac{1}{2}\mathbb{Z}^{2}$. 
\end{theorem}
\begin{proof}
    It suffices to prove the necessity. Suppose that the generalized Lam\'{e}-type equation \eqref{GLE1,even}$_{\tau,p,A}$ is completely reducible. Then $Q(A)\neq0$, or equivalently, the two Baker-Akhiezer functions $\psi(P;z)$ and $\psi(P^{*};z)$ are linearly independent. 
    
    Assume, for contradiction, that $(r(P),s(P))\in\tfrac{1}{2}\mathbb{Z}^{2}$. Then
\[
\bigl(\lambda_{1}(P),\lambda_{2}(P)\bigr)\in\{(1,1),(1,-1),(-1,1),(-1,-1)\},
\]
and hence the function
\[
\psi^{2}(P;z)+\psi^{2}(P^{*};z)
\]
is an elliptic solution of \eqref{3rd ODE}. By Theorem \ref{unique elliptic sol, 3rd ode, even} and equation \eqref{0907equ4},
\[
\psi(P;z)\,\psi(P^{*};z)=\Phi_{e}(z;A) =\psi^{2}(P;z)+\psi^{2}(P^{*};z),
\]
up to a nonzero multiple. Thus $\psi(P;z)$ and $\psi(P^{*};z)$ have common zeros, contradicts to their linear independence.
\end{proof}

If the generalized Lam\'{e}-type equation \eqref{GLE1,even}$_{\tau,p,A}$ is completely reducible, the Baker-Akhiezer functions $\psi(P;z)$ and $\psi(P^{*};z)$ form a fundamental system with monodromy matrices:
\[
M_1(p,A)=
\left(\begin{array}{cc}
    e^{-2\pi i s(P)} & 0 \\
    0 & e^{2\pi i s(P)}
\end{array}\right),~
M_2(p,A)=
\left(\begin{array}{cc}
    e^{2\pi i r(P)} & 0 \\
    0 & e^{-2\pi i r(P)}
\end{array}\right).
\]
Thus, up to a $\pm$ sign and modulo $\mathbb{Z}^{2}$, the pair $(r(P),s(P))\in\mathbb{C}^{2}\setminus \tfrac{1}{2}\mathbb{Z}^{2}$ is the monodromy data of \eqref{GLE1,even}$_{\tau,p,A}$.

If, in addition, the monodromy matrices $M_j(p,A)$, $j=1,2$, are unitary, then $(r(P),s(P))\in\mathbb{R}^{2}\setminus \tfrac{1}{2}\mathbb{Z}^{2}$.
By Theorem \ref{pde to ode}, the curvature equation \eqref{Curvature 1}$_{\tau,p}$ admits an even solution $u(z)$. 
As in the sufficiency part of its proof, the developing map $f(z)$ associated with $u(z)$ can be chosen as
\begin{equation}\label{developing map written as BA/BA*}
    f(z)=\dfrac{\psi(P;z)}{\,\psi(P^*;z)\,}.
\end{equation}
From \eqref{0907equ8} and \eqref{definition of r(P),s(P)}, $f(z)$ is a type II developing map satisfying
\begin{equation}
    f(z+\omega_1)=e^{-4\pi is(P)},\quad
    f(z+\omega_2)=e^{4\pi ir(P)}.
\end{equation}
Therefore, by \eqref{type II}, $(r(P),s(P))$ is also the monodromy data of the corresponding even family of solutions of curvature equation \eqref{Curvature 1}$_{\tau,p}$. We summarize these as follows.

\begin{proposition}\label{two types of monodromy data}
    Suppose the curvature equation \eqref{Curvature 1}$_{\tau,p}$ admits an even solution $u(z)$ associated with the generalized Lam\'{e}-type equation \eqref{GLE1,even}$_{\tau,p,A}$. Then the monodromy data of this even solution, in the sense of \eqref{type II}, coincides with the monodromy data $(r(P),s(P))\in\mathbb{R}^2\setminus\frac{1}{2}\mathbb{Z}^2$ of equation \eqref{GLE1,even}$_{\tau,p,A}$, where $P=(A,C)\in\Gamma(\tau,p)$. Moreover, the  type II developing map $f(z)$ of $u(z)$ is given by the ratio of the two Baker-Akhiezer functions given in \eqref{developing map written as BA/BA*}.
\end{proposition}

To locate the blow-up points of an even family of solutions to \eqref{Curvature 1}$_{\tau,p}$, we next analyze the zeros of the Baker-Akhiezer functions.

Since the Baker-Akhiezer function $\psi(P;z)$ is an elliptic function of the second kind, it could be written as:
\[
\psi(P;z)=e^{c(P)z}\,\frac{\,\sigma(z-a_1(P))\,\sigma(z-a_2(P))\,}{\sigma(z)\sqrt{\sigma(z+p)\,\sigma(z-p)}},
\]
where $c(P)\in\mathbb{C}$ and $a_1(P),a_2(P)\in E_{\tau}\setminus\{0\}$ are the two zeros of $\psi(P;z)$. The following proposition establishes algebraic relations between $(r(P),s(P))$ and $(a_1(P),a_2(P),c(P))$.

\begin{proposition}\label{0807prop1} 
With the above notations,
\begin{equation}
\label{algebraic relations between a1(P),a2(P),r(P),s(P),c(P)}%
\begin{cases}
~r(P)+s(P)\tau= a_{1}(P)+a_{2}(P),\\[5pt]%
~r(P)\,\eta_{1}(\tau)+s(P)\,\eta_{2}(\tau)=c(P).
\end{cases}
\end{equation}
\end{proposition}
\begin{proof}
From (\ref{transformation law of the Weierstrass sigma function}), we obtain
\[
\psi(P;z+\omega_{j})=e^{c(P)\omega_{j}-\eta_{j}\left(  a_{1}(P)+a_{2}%
(P)\right)  }\psi(P;z),\quad\text{for }\,j=1,2.
\]
Comparing with \eqref{0907equ8} and \eqref{definition of r(P),s(P)}, we deduce
\begin{equation}%
\begin{cases}
~c(P)-\eta_{1}\left(  a_{1}(P)+a_{2}(P)\right)  =-2\pi i\,s(P),\\
~c(P)\tau-\eta_{2}\left(  a_{1}(P)+a_{2}(P)\right)  =2\pi i\,r(P).
\end{cases}
\label{0807equ1}%
\end{equation}
Therefore, \eqref{algebraic relations between a1(P),a2(P),r(P),s(P),c(P)} follows from the Legendre relation
\begin{equation}
\tau\eta_{1}-\eta_{2}=2\pi i.\label{Legendre relation}%
\end{equation}
\end{proof}

In what follows, we derive algebraic equations satisfied by zeros $a_1(P),a_2(P)$. To this end, define
\begin{equation}
\label{ya}y_{a,c}(z):=e^{cz}\dfrac{\sigma(z-a_{1}) \,\sigma(z-a_{2})}
{\,\sigma(z)\sqrt{\sigma(z+p)\,\sigma(z-p)}\,},
\end{equation}
where $c\in\mathbb{C}$, $a_{1}, a_{2}\in E_{\tau}\setminus\{0\}$. 
\normalsize

\begin{theorem}
\label{a1,a2 alge equ} Let $y_{a,c}(z)$ be defined in \eqref{ya}.
\begin{enumerate}
\item[(i)] Suppose $a_{1}\neq a_{2}\in E_{\tau}\setminus\{0,\pm p\}$.
Then $y_{a,c}(z)$ is a solution to \eqref{GLE1,even}$_{\tau,p,A}$ for some
$A\in\mathbb{C}$, with $B$ given by \eqref{B in even potential}, if and only if
\begin{equation}
\label{a1 and a2's condition}
\dfrac{\wp'(a_1)}{\,\wp(p)-\wp(a_1)\,}+\dfrac{\wp'(a_2)}{\,\wp(p)-\wp(a_2)\,}=0.
\end{equation}
In this case, the constants $c$ and $A$ are determined by
\begin{align}
c  &  =\zeta(a_{1})+\zeta(a_2),\label{c determined by a1, a2}\\
A  &  =\dfrac{\wp'(p)-\wp'(a_1)}{\,2\left(\wp(p)-\wp(a_1)\right)\,}+\dfrac{\wp'(p)-\wp'(a_2)}{\,2\left(\wp(p)-\wp(a_2)\right)\,}-\dfrac{\wp''(p)}{\,4\wp'(p)\,}. \label{A determined by a1, a2}%
\end{align}

\item[(ii)] Suppose $a_{1}=a_{2}\in\{\pm p\}$. Then $y_{a,c}(z)$ is a solution
to \eqref{GLE1,even}$_{\tau,p,A}$ for some $A\in\mathbb{C}$, with $B$ given by
\eqref{B in even potential}, if and only if
\begin{align}
c  &  =\pm2\zeta(p),\label{c determined by a1=a2=pm p}\\
A  &  =\dfrac{\,3\wp^{\prime\prime}(p)\,}{4\wp^{\prime}(p)}.
\label{A determined by a1=a2=pm p}%
\end{align}
\end{enumerate}
\end{theorem}

From the expression of $y_{a,c}(z)$ in \eqref{ya}, we have
\begin{align*}
\dfrac{y_{a,c}^{\prime}(z)}{y_{a,c}(z)}  &  =c+\zeta(z-a_{1})+\zeta(z-a_{2}%
)-\zeta(z)-\dfrac{1}{2}\left(  \zeta(z+p)+\zeta(z-p)\right)  ,\\
\left(  \dfrac{y_{a,c}^{\prime}(z)}{y_{a,c}(z)}\right)  ^{\prime}  &
=\wp(z)+\dfrac{1}{2}\left(  \wp(z+p)+\wp(z-p)\right)  -\wp(z-a_{1}
)-\wp(z-a_{2}).
\end{align*}
Define the elliptic function
\begin{equation}\label{0905equ11}
g(z):=\left(  \dfrac{y_{a,c}^{\prime}(z)}{y_{a,c}(z)}\right)  ^{\prime}+\left(
\dfrac{y_{a,c}^{\prime}(z)}{y_{a,c}(z)}\right)^{2}\!-q(z;A).
\end{equation}
Then $y_{a,c}(z)$ solves \eqref{GLE1,even}$_{\tau,p,A}$ if and only if $g(z)\equiv0$.

\begin{proof}[Proof of Theorem \ref{a1,a2 alge equ}]
\textbf{Case (i).} Suppose $a_{1}\neq a_{2}\in E_{\tau}\setminus\{0,\pm p\}$.

The expansion of $g(z)$ at $z=0$ gives
\begin{equation}
\label{0905equ1}%
\begin{split}
g(z)=\dfrac{\,2\left(\zeta(a_1)+\zeta(a_2)-c\right)\,}{z}+O(1),
\end{split}
\end{equation}
so $g$ is holomorphic at $z=0$ if and only if 
\begin{equation}\label{0905equ2}
c=\zeta(a_{1})+\zeta(a_{2}).
\end{equation}
Under \eqref{0905equ2}, the expansion at $z=p$ yields
\begin{equation}\label{0905equ3}
\begin{split}
    g(z)&=\dfrac{1}{z-p}\left(A-\dfrac{\wp'(p)+\wp'(a_1)}{\,2\left(\wp(p)-\wp(a_1)\right)\,}-\dfrac{\wp'(p)+\wp'(a_2)}{\,2\left(\wp(p)-\wp(a_2)\right)\,}+\dfrac{\wp''(p)}{\,4\wp'(p)\,}\right)\\[2pt]
    &-A^2+\left(\dfrac{\wp'(p)+\wp'(a_1)}{\,2\left(\wp(p)-\wp(a_1)\right)\,}+\dfrac{\wp'(p)+\wp'(a_2)}{\,2\left(\wp(p)-\wp(a_2)\right)\,}-\dfrac{\wp''(p)}{\,4\wp'(p)\,}\right)^2\!\!\!+O(z-p).
\end{split}
\end{equation}
Thus $g$ is holomorphic at $z=p$ with $g(p)=0$ if and only if
\begin{equation}\label{0905equ4}
    A=\dfrac{\wp'(p)+\wp'(a_1)}{\,2\left(\wp(p)-\wp(a_1)\right)\,}+\dfrac{\wp'(p)+\wp'(a_2)}{\,2\left(\wp(p)-\wp(a_2)\right)\,}-\dfrac{\wp''(p)}{\,4\wp'(p)\,}.
\end{equation}
Similarly, the expansion at $z=-p$ (still under \eqref{0905equ2}) gives
\begin{equation}\label{0905equ5}
    g(z)=\dfrac{-1}{z+p}\left(A-\dfrac{\wp'(p)-\wp'(a_1)}{\,2\left(\wp(p)-\wp(a_1)\right)\,}-\dfrac{\wp'(p)-\wp'(a_2)}{\,2\left(\wp(p)-\wp(a_2)\right)\,}+\dfrac{\wp''(p)}{\,4\wp'(p)\,}\right)+O(1),
\end{equation}
hence holomorphicity at $z=-p$ requires
\begin{equation}\label{0905equ6}
    A=\dfrac{\wp'(p)-\wp'(a_1)}{\,2\left(\wp(p)-\wp(a_1)\right)\,}+\dfrac{\wp'(p)-\wp'(a_2)}{\,2\left(\wp(p)-\wp(a_2)\right)\,}-\dfrac{\wp''(p)}{\,4\wp'(p)\,}.
\end{equation}
Comparing \eqref{0905equ4} and \eqref{0905equ6} yields
\begin{equation}\label{0905equ7}
    \dfrac{\wp'(a_1)}{\,\wp(p)-\wp(a_1)\,}+\dfrac{\wp'(a_2)}{\,\wp(p)-\wp(a_2)\,}=0.
\end{equation}
The Laurent expansion at $z=a_1$ gives
\begin{equation}\label{0905equ7*}
    g(z)=\dfrac{1}{z-a_1}\left(\dfrac{\wp'(a_1)\left(\wp(p)-\wp(a_2)\right)+\wp'(a_2)\left(\wp(p)-\wp(a_1)\right)}{\left(\wp(p)-\wp(a_1)\right)\left(\wp(a_1)-\wp(a_2)\right)}\right)+O(1).
\end{equation}
Under \eqref{0905equ2}, \eqref{0905equ6}, and \eqref{0905equ7}, the only possible poles of $g$ lie at $a_1,a_2$. By \eqref{0905equ7*}, it follows that $g$ is holomorphic at $a_1$ if and only if
\[
\wp'(a_1)\left(\wp(p)-\wp(a_2)\right)+\wp'(a_2)\left(\wp(p)-\wp(a_1)\right)=0,
\]
which is equivalent to \eqref{0905equ7}. Since $g$ is elliptic and regular at $0,\pm p$, thus vanishing of the residue at $a_1$ forces regularity at $a_2$ as well. In conclusion, $g(z)\equiv0$ if and only if \eqref{0905equ2}, \eqref{0905equ6}, and \eqref{0905equ7} hold. This proves \textit{(i)}.
\vspace{3pt}
\newline\textbf{Case (ii).}
Suppose $a_{1}=a_{2}\in\{\pm p\}$. Since $q(z;A)$ is even, equation \eqref{GLE1,even}$_{\tau,p,A}$ is invariant under $z\mapsto-z$. Hence it suffices to consider $a_1=a_2=p$, and the case $a_1=a_2=-p$ follows by symmetry.

Expanding at $z=0$ gives
\[
g(z)=\dfrac{\,-2\left(  c-2\zeta( p)\right)  \,}{z}+O(1),
\]
hence
\begin{equation}
\label{0905equ9}c=2\zeta(p).
\end{equation}
Expanding further at $z=p$ yields
\begin{equation*}
\begin{split}
    g(z)&=\dfrac{1}{z-p}\left(A-\dfrac{3\wp''(p)}{\,4\wp'(p)\,}\right)
    -A^2+\dfrac{9\wp''(p)^2}{\,16\wp'(p)^2\,}+O(z-p).
\end{split}
\end{equation*}
which forces
\begin{equation}\label{0905equ10}
A=\dfrac{3\wp''(p)}{\,4\wp'(p)\,}.
\end{equation}
Therefore, $y_{a,c}(z)$ is a solution to \eqref{GLE1,even}$_{\tau,p,A}$ iff \eqref{0905equ9} and \eqref{0905equ10} hold.
\end{proof}
	Since this paper focuses on even families of cone spherical metrics blowing up at a conical point $p$, the associated developing map $f(z)$ must have a double zero at $p$ and a double pole at $-p$, corresponding precisely to case (ii) of Theorem \ref{a1,a2 alge equ}. According to Theorem \ref{a1,a2 alge equ}-(ii), the parameter $A$ is given by
		\[
		A=\dfrac{3\wp''(p)}{\,4\wp'(p)\,}.
		\]
		Substituting this into the spectral polynomial $Q(A)$ yields
		\[
		Q\left(\dfrac{3\wp''(p)}{\,4\wp'(p)\,}\right)=-\wp^{\prime}(p)^{2}\neq0,
		\]
		since $p\notin E_{\tau}[2]$. It then follows from Theorem \ref{Q(A) neq 0 iff com.red.} that the generalized Lam\'{e}-type equation \eqref{GLE1,even}$_{\tau,p,A}$ with parameter $A=\frac{3\wp''(p)}{\,4\wp'(p)\,}$ is completely reducible.

	In the following, we set 
	\begin{equation}\label{A0,even}
		A_0:=\dfrac{3\wp''(p)}{\,4\wp'(p)\,},
	\end{equation}
	and denote by 
	\begin{equation}\label{r0,s0,definition}
		(r_0,s_0)\in\mathbb{C}^2\setminus\frac{1}{2}\mathbb{Z}^2
	\end{equation}
	the monodromy data of the generalized Lam\'{e}-type equation \eqref{GLE1,even}$_{\tau,p,A_0}$. We summarize the above discussion as follows.
	
	\begin{proposition}\label{Main Proposition}
		Assume $p\in E_\tau\setminus E_\tau[2]$, and adopt the above notations \eqref{A0,even} and \eqref{r0,s0,definition}. Then the generalized Lam\'{e}-type equation \eqref{GLE1,even}$_{\tau,p,A_0}$ with  is completely reducible. Moreover,
		\begin{enumerate}
			\item[(i)] The Baker-Akhiezer functions $\psi(P;z)$ and $\psi(P^*;z)$ are given explicitly by
			\begin{align}
				\psi(P;z)&=\dfrac{\,e^{2\zeta(p)z}\,\sigma^{\frac{3}{2}}(z-p)\,}{\sigma(z)\,\sigma^{\frac{1}{2}}(z+p)},\label{BA when a1=a2=p, even}\\[3pt]
				\psi(P^*;z)&=\psi(P;-z).\label{BA when a1=a2=-p, even}
			\end{align}
			\item[(ii)] The monodromy data $(r_0,s_0)\in\mathbb{C}^2\setminus\frac{1}{2}\mathbb{Z}^2$ satisfy
			\begin{equation}\label{r0,s0,2p,relation}
				\begin{cases}
					r_0+s_0\tau=2p,\\[2pt]
					r_0\,\eta_1(\tau)+s_0\,\eta_2(\tau)=2\zeta(p).
				\end{cases}
			\end{equation}
		\end{enumerate}
	\end{proposition}

\section{The Proofs}
\label{The Proofs}

For any $\left(  r,s\right)  \in\mathbb{C}^{2}\setminus\frac{1}{2}%
\mathbb{Z}^{2}$, we define
\begin{equation}
Z(r,s,\tau):=\zeta(r+s\tau)-r\eta_{1}(\tau)-s\eta_{2}(\tau),\label{zrs}%
\end{equation}
which depends meromorphically on $(r,s)$. Since $(r,$ $s)$ $\not \in $
$\frac{1}{2}\mathbb{Z}^{2}$, it is clear that $Z(r,s,\tau)$ $\not \equiv
0,\infty$ as a function in $\tau$, and it remains meromorphic in $\tau$. This
function was introduced by Hecke in \cite{Heck}, who showed that $Z(r,s,\tau)$
is a modular form of weight one with respect to 
\[
\Gamma(n):=\left\{\gamma\in SL(2,\mathbb{Z})~|~\gamma\equiv I_2\,\operatorname{mod}n\right\}
\]
whenever $(r,$
$s)$ is an $n$-torsion point.

For fixed $\tau$, any point $z\in E_{\tau}$ can be written uniquely as
\[
z=r+s\tau,\quad\left(  r,s\right)  \in\mathbb{R}^{2}\text{ mod
}\mathbb{Z}^{2}\text{.}%
\]
Define the linear map $\eta:E_{\tau}\rightarrow\mathbb{C}$, $z\longmapsto
\eta(z)$ by
\[
\eta(z):=r\eta_{1}(\tau)+s\eta_{2}(\tau)\text{ \,if\, }z=r+s\tau,\quad%
r,s\in\mathbb{R}.
\]
With this notation, the function $Z(r,s,\tau)$ induces a meromorphic function
on $E_{\tau}$, defined by%
\[
Z(z;\tau):=\zeta(z;\tau)-\eta(z),
\]
which satisfies%
\[
Z(z;\tau)=Z(r,s,\tau)\text{, \,where\, }z=r+s\tau,\text{ ~}r,s\in\mathbb{R}^{2}.
\]
The following identity was proved in \cite{LW-AnnMath}:
\[
-4\pi\partial_{z}G(z;\tau)=Z(z;\tau),
\]
Consequently, for each $a\in E_\tau\setminus E_\tau[2]$, we have 
\begin{equation}\label{0910equ1}
\partial_zG(a;\tau)=0~\Longleftrightarrow~ Z(a;\tau)=0.
\end{equation}
Let
\begin{align*}
\Delta_{0}&:=\left\{  \left(  r,s\right)  ~\left\vert~ 0<r,s<\frac{1}{2},~r+s>\frac{1}{2}\right.  \right\}\subset[0,1/2]\times[0,1]\setminus\frac{1}{2}\mathbb{Z}^2,\\[3pt]
F_0&:=\left\{\tau\in\mathbb{H}~\bigg\vert~0\leqslant\operatorname{Re}\tau\leqslant1,~\left|\tau-\dfrac{1}{2}\right|\geqslant\dfrac{1}{2}\right\}.
\end{align*}
As a consequence, $G(z;\tau)$ admits a nontrivial critical point $a\in E_\tau\setminus E_\tau[2]$ for some $\tau\in F_0$ if and only if $(r,s)\in\Delta_0$. Moreover, for each such $(r,s)$, the corresponding $\tau\in F_0$ is uniquely determined \cite[Theorem~1.3]{Chen-Kuo-Lin-Wang-JDG}. 

Let
\[
\Omega:=\left\{  \tau\in F_{0}~\left\vert~ Z(r,s,\tau)=0\text{ for some }\left(r,s\right)  \in\Delta_{0}\right.  \right\}  .
\]
This induces a real-analytic bijection%
\begin{equation}\label{tau map}
\begin{split}
\tau:\Delta_{0}\longrightarrow\Omega,\quad
(r,s)\longmapsto\tau(r,s),\text{\, with\, }Z(r,s,\tau(r,s))=0.
\end{split}
\end{equation}
In the following, denote by $\tau(r,s)\in\Omega$ the unique parameter associated with $(r,s)\in\Delta_{0}$.

The geometry of $\Omega$ is characterized as a triangular region within $F_0$, bounded by three smooth curves $C_i$ for $i=1,2,3$. These curves correspond to tori where the Green function degenerates at $\frac{\omega_i}{2}$; refer to Figure \ref{Omega}.

\begin{theorem}\label{section4, main thm}
	Adopt the notations \eqref{A0,even} and \eqref{r0,s0,definition}. Then the curvature equation \eqref{Curvature 1}$_{\tau,p}$ admits an even family of solutions $\{\mathbf{u}_\beta(z;\tau,p)\}$, blowing up at the cone singularities $\pm p$ as $\beta\to\pm\infty$ if and only if 
    the generalized Lam\'{e}-type equation \eqref{GLE1,even}$_{\tau,p,A_0}$
    has unitary monodromy.
	In this case, writing $p=r+s\tau$ with $(r,s)\in\mathbb{R}^{2}\setminus \tfrac{1}{2}\mathbb{Z}^{2}$, the following hold:
	\begin{enumerate} 
		\item [(i)] $(r,s)\in\Delta_0$, $\tau=\tau(r,s)\in\Omega$. That is, $Z(r,s,\tau)=0$, equivalently, the Green function $G(z;\tau)$ admits a pair of nontrivial critical points at $\pm p$.
		\item[(ii)] The monodromy data of \eqref{GLE1,even}$_{\tau,p,A_0}$ are given by $(2r,2s)\in\mathbb{R}^2\setminus\frac{1}{2}\mathbb{Z}^2$.
        \item [(iii)] The type II developing map can be expressed explicitly as
		    \[
		    f(z)=\frac{\psi(P;z)}{\psi(P^*;z)}=e^{4\zeta(p)z}\,\frac{\sigma(z-p)^2}{\,\sigma(z+p)^2\,}\quad\text{where},
		    \]
        $$P=\left(\frac{3\wp''(p)}{\,4\wp'(p)\,},\,i\wp^{\prime}(p)\right)\in\Gamma(\tau,p).$$
          \item [(iv)] The even family $\{\mathbf{u}_\beta(z;\tau,p)\}$ of equation \eqref{Curvature 1}$_{\tau,p}$, blowing up at $\pm p$ as $\beta\to\pm\infty$, is unique.
	\end{enumerate}
\end{theorem}
\newpage

\begin{figure}[htp]
    \centering
    \includegraphics[width=\textwidth]{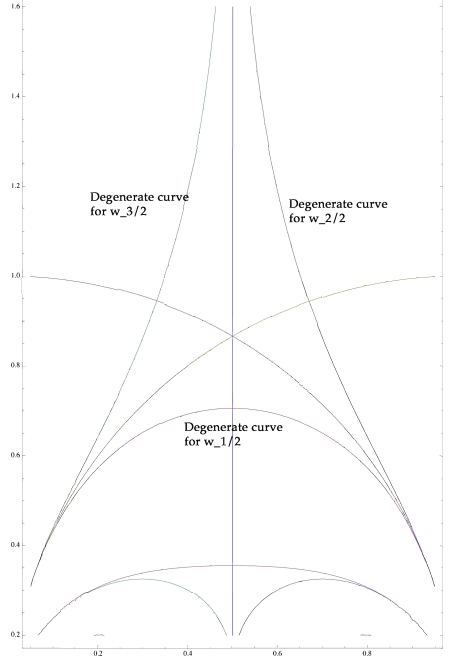}
    \caption{The domain $\Omega$.}
    \label{Omega}
\end{figure}
\newpage

\begin{proof}[Proof of Theorem \ref{section4, main thm}]
Suppose the curvature equation \eqref{Curvature 1}$_{\tau,p}$ admits an even family $\{\mathbf{u}_{\beta}(z;\tau,p)\}$, blowing up at $p$ as $\beta\to\infty$ and at $-p$ as $\beta\to-\infty$. By Theorem \ref{pde to ode}, there exists $A\in\mathbb{C}$ such that the associated equation \eqref{GLE1,even}$_{\tau,p,A}$ has unitary monodromy. Moreover, its type II developing map $f(z)$ is given by \eqref{developing map written as BA/BA*}:
\begin{equation}\label{0911equ1}
    f(z)=\dfrac{\psi(P;z)}{\,\psi(P^*;z)\,}=\dfrac{\psi(P;z)}{\,\psi(P;-z)\,}.
\end{equation}

Since this even family blows up at $\pm p$ as $\beta\to\pm\infty$, we deduce that $p$ must be a zero of $\psi(P;z)$. Consequently, $a_1(P)=a_2(P)=p$. Namely, Theorem \eqref{a1,a2 alge equ}-(ii) applies, which yields
\[
A=A_0=\dfrac{3\wp''(p)}{4\wp'(p)}.
\]
The sufficiency follows from Theorem \ref{pde to ode} and Theorem \ref{a1,a2 alge equ}-(ii) by taking
\begin{equation}\label{0911equ2}
f(z)=\dfrac{\psi(P;z)}{\,\psi(P^*;z)\,}=e^{4\zeta(p)z}\,\dfrac{\sigma^2(z-p)}{\,\sigma^2(z+p)\,}
\end{equation}
and 
\begin{equation}\label{0911equ3}
u(z):=\log\dfrac{8\left|f'(z)\right|^2}{\,\left(1+|f(z)|^2\right)^2\,}.
\end{equation}
Clearly, \textit{(iii)} follows from \eqref{0911equ2}. We now establish \textit{(i)} and \textit{(ii)}. 

By Proposition \ref{Main Proposition}, the monodromy data $(r_0,s_0)\in\mathbb{R}^2\setminus\frac{1}{2}\mathbb{Z}^2$ of \eqref{GLE1,even}$_{\tau,p,A_0}$ satisfy
        \[
        \begin{cases}
            r_0+s_0\tau=2p,\\[2pt]
            r_0\,\eta_1(\tau)+s_0\,\eta_2(\tau)=2\zeta(p).
        \end{cases}
        \]
Since $p=r+s\tau$ with $r,s\in\mathbb{R}$, the first equation gives 
       \[
       r_0=2r\quad\text{and}\quad s_0=2s.
       \]
       Thus, $(2r,2s)$ is the monodromy data of \eqref{GLE1,even}$_{\tau,p,A_0}$. Moreover,
       \begin{align*}
           Z(r,s,\tau)&=\zeta(r+s\tau)-r\,\eta_1(\tau)-s\,\eta_2(\tau)\\&=
           \zeta(p)-\dfrac{1}{2}\left(r_0\,\eta_1(\tau)+s_0\,\eta_2(\tau)\right)\\&=
           \zeta(p)-\dfrac{1}{2}\cdot2\zeta(p)=0.
       \end{align*}
       Finally, for \textit{(iv)}: if $\{\mathbf{u}_{\beta}(z;\tau,p)\}$ is an even family, blowing up at $\pm p$ as $\beta\to\pm\infty$, then the parameter $A$ of generalized Lam\'{e}-type equation \eqref{GLE1,even}$_{\tau,p,A}$ must be $A_0$. Moreover, its type II developing map must be given by \eqref{0911equ2}. Therefore, in view of \eqref{0911equ3}, $\textbf{u}_{\beta}(z;\tau,p)$ must be determined as
       \[
       \textbf{u}_{\beta}(z;\tau,p)=\log\dfrac{8e^{2\beta}\left|f'(z)\right|^2}{\,\left(1+e^{2\beta}|f(z)|^2\right)^2\,},\quad \forall\,\beta\in\mathbb{R}.
       \]
       This establishes the uniqueness of such even families.
	\end{proof}

\begin{proof}[Proof of Theorem \ref{Main theorem1}]
The necessity follows immediately from Theorem \ref{section4, main thm} together with \eqref{0910equ1}.

For sufficiency, suppose that the Green function $G(z;\tau)$ admits a pair of nontrivial critical points $\pm a\in E_\tau\setminus E_\tau[2]$. Writing $a=r+s\tau$ with $(r,s)\in \mathbb{R}^{2}\setminus \tfrac{1}{2}\mathbb{Z}^{2}$, \eqref{0910equ1} gives
\begin{equation}\label{0910equ3}
Z(r,s,\tau)=0.
\end{equation}

Consider the generalized Lam\'e-type equation \eqref{GLE1,even}$_{\tau,a,A_0}$.
By Proposition \ref{Main Proposition}, the monodromy data $(r_0,s_0)\in\mathbb{C}^2\setminus\frac{1}{2}\mathbb{Z}^2$ of equation \eqref{GLE1,even}$_{\tau,a,A_0}$ satisfy:
\begin{equation}\label{0910equ2}
        \begin{cases}
            r_0+s_0\tau=2a,\\[2pt]
            r_0\,\eta_1+s_0\,\eta_2=2\zeta(a).
        \end{cases}
        \end{equation}
Using \(a=r+s\tau\) and \eqref{0910equ3}, the second line of \eqref{0910equ2} yields
\begin{align*}
    r_0\,\eta_1+s_0\,\eta_2&=2\zeta(r+s\tau)\\&=
    2\zeta(r+s\tau)-2Z(r,s,\tau)\\&=2r\eta_1+2s\eta_2.
\end{align*}
We obtain the system
\[
\begin{cases}
            r_0+s_0\tau=2r+2s\tau,\\[2pt]
            r_0\,\eta_1+s_0\,\eta_2=2r\,\eta_1+2s\,\eta_2.
\end{cases}
\]
By the Legendre relation \eqref{Legendre relation}, this implies
\begin{equation}\label{0910equ80}
(r_0,s_0)=(2r,2s)\in\mathbb{R}^2\setminus\frac{1}{2}\mathbb{Z}^2.
\end{equation}
Therefore, by Theorem \ref{section4, main thm}, the curvature equation \eqref{Curvature 1}$_{\tau,a}$ admits an even family of solutions $\{\mathbf{u}_\beta(z;\tau,a)\}$, blowing up at the cone singularities $\pm a$ as $\beta\to\pm\infty$. This completes the proof.
\end{proof}

Theorem \ref{Main theorem2} follows from Proposition \ref{two types of monodromy data} and \eqref{0910equ80}.

\begin{proof}[Proof of Theorem \ref{Main theorem4}]
Suppose the Green function $G(z;\tau)$ has a pair of nontrivial critical points $\pm a$. According to Theorem \ref{section4, main thm}-(ii), the type II developing map $f(z;\tau,a)$ associated to $\{u_{\beta}(z;\tau,a)\}$ is given by
		\begin{equation}
			f(z;\tau,a)=e^{4\zeta(a)z}\frac{\sigma(z-a)^2}{\sigma(z+a)^2}.  \label{type II of u(z)}
		\end{equation}
		It was show in \cite[(3.16)]{LW-AnnMath} that the type II developing map $h(z;\tau)$ associated to $\{v_{\beta}(z;\tau)\}$ is given by
		\[
h(z;\tau)=\exp\left(\int^{z}\frac{\wp^{\prime}(a)}{\wp(\xi)-\wp(a)}d\xi\right),
		\]
		or equivalently,
		\begin{equation}
			h(z;\tau)=e^{2\zeta(a)z}\,\frac{\sigma(z-a)}{\,\sigma(z+a)\,}.  \label{type II of v(z)}
		\end{equation}
		By \eqref{type II of u(z)} and \eqref{type II of v(z)}, we obtain
		\[
		f(z;\tau,a)=h(z;\tau)^2.
		\]

\end{proof}

\begin{proof}[Proof of Theorem \ref{Main theorem3}]
   Assume the Green function 
   $G(z;\tau)$ admits a pair of nontrivial critical points $\pm a$ satisfying \eqref{nontrivial critical point}. Let 
$\left\{\mathfrak{u}_\beta(z;\tau,a)\right\}_{(r,s)}$ be the solution family given by Theorem 3.3 in \cite{KLW}, and let  $\left\{\mathbf{u}_\beta(z;\tau,a)\right\}_{(2r,2s)}$ denote the even family of solutions to equation \eqref{Curvature 1}$_{\tau,a}$, as established in Theorem \ref{Main theorem1}.

For the necessary part,  suppose that for the curvature equation \eqref{curvature special}, the following two families coincide:
		\begin{equation}
		    \left\{\mathbf{u}_\beta(z;\tau,a)\right\}_{(2r,2s)}=\left\{\mathfrak{u}_\beta(z;\tau,a)\right\}_{(r,s)},
            \label{two family  equal}
		\end{equation}
        where  $a$ is a cone point.
        
		Since $\left\{\mathfrak{u}_\beta(z;\tau,a)\right\}_{(r,s)}$ is an even family of equation \eqref{curvature special}, it is known in \cite{Chen-Kuo-Lin-Painleve VI} that $\wp(a)$ can be expressed as
		\begin{equation}
			\wp(a)=\wp(a)+\frac{3\wp^{\prime}(a)Z(a;\tau)+2\wp^{\prime\prime}(a)Z(a;\tau)+3\wp(a)\wp^{\prime}(a)}{2(Z(a;\tau)^3-3\wp(a)Z(a;\tau)-\wp^{\prime}(a))}. \label{Hitch's formula}
		\end{equation}
		From $Z(a;\tau)=0$ and \eqref{Hitch's formula}, we have 
        $$\wp(a)=0.$$ 
        
        Before going on, we recall  Theorem 3.5 in \cite{KLW}: The family $\left\{\mathfrak{u}_\beta(z;\tau,a)\right\}_{(r,s)}$ blows up at the cone point $a$ as $\beta\to +\infty$ if and only if $2a=\pm (r+s\tau)$ $\mod\, \Lambda_{\tau}$.
        
        By \eqref{two family  equal}, $\left\{\mathfrak{u}_\beta(z;\tau,a)\right\}_{(r,s)}$ blows up at $a$ as $\beta\to+\infty$,  we have $2a\equiv\pm a\, \mod \Lambda_{\tau}$,
		which implies 
        $$\wp(2a)=\wp(a)=0.$$
        
        By the addition formula, 
		\[
		0=\wp(2a)=-2\wp(a)+\frac{\wp^{\prime\prime}(a)^2}{4\wp^{\prime}(a)^2}=\frac{\wp^{\prime\prime}(a)^2}{4\wp^{\prime}(a)^2},
		\]
	we have
		\[\wp^{\prime\prime}(a)=0.\]
		Since $\wp(a)=\wp^{\prime\prime}(a)=0$, we obtain
		 \[
		 g_2(\tau)=0,
		 \]
		which implies 
        $$\tau=\rho=e^{\pi i/3}.$$
		Therefore,
\[
(r,s)=\left(\frac{1}{3},\frac{1}{3}\right).
\]
       This proves the necessary part.
        
        For the sufficient part, let $\tau=\rho$ and $a=\frac{1+\rho}{3}$. Recall that 
        \[\wp\left(a;\rho\right)=\wp^{\prime\prime}\left(a;\rho\right)=Z(a;\tau)=0.
        \]
It is easy to verify $\wp(a)$ satifies \eqref{Hitch's formula}. Therefore, the cone point  $a$ satisfies Theorem 3.3-(ii) in \cite{KLW}, which implies the family $\left\{\mathfrak{u}_\beta(z;\tau,a)\right\}_{(\frac{1}{3},\frac{1}{3})}$ is an even family of solutions to the curvature equation
\[
\Delta u+e^u=8\pi\delta_0+4\pi\left(\delta_{\frac{1+\rho}{3}}+\delta_{-\frac{1+\rho}{3}}\right)\quad\text{on }E_{\rho}.
\]

       Moreover,  since
		\[
		2a=\frac{2(1+\rho)}{3}\equiv-\frac{1+\rho}{3}=-a,
		\]
		it follows from Theorem 3.5 in \cite{KLW}, as recalled in above, that $\left\{\mathfrak{u}_\beta(z;\tau,a)\right\}_{\left(\frac{1}{3},\frac{1}{3}\right)}$ blows up at $a=\frac{1+\rho}{3}$ as $\beta\to+\infty$. By the evenness of $\left\{\mathfrak{u}_\beta(z;\tau,a)\right\}_{\left(\frac{1}{3},\frac{1}{3}\right)}$, it blows up at $\pm a$ as $\beta\to\pm\infty$, respectively.
        
        From the uniqueness of such an even family of solutions in Theorem \ref{section4, main thm}, we conclude that
		\[
		\left\{\mathbf{u}_\beta(z;\tau,a)\right\}_{(\frac{2}{3},\frac{2}{3})}=\left\{\mathfrak{u}_\beta(z;\tau,a)\right\}_{(\frac{1}{3},\frac{1}{3})}.
		\]
        This proves the sufficient part.
\end{proof}

\end{document}